\documentclass[final,1p,times]{elsarticle}
	\usepackage{graphicx}
\usepackage{amssymb}
\usepackage{amsthm}
\usepackage{url}
\usepackage{lineno}
 
\usepackage{amssymb}
\usepackage{amsmath}
\usepackage{makeidx}
\usepackage{amscd}
\usepackage{fancyhdr}
\usepackage{yfonts}
\usepackage{graphicx}
\usepackage{bbm}
\usepackage{enumitem}
\usepackage{amsthm}
\usepackage{listings}
\usepackage{fancyvrb}
\usepackage{url}
\usepackage{mathrsfs}  
\usepackage[ruled]{algorithm2e}
\usepackage{color}
\usepackage{comment}
\usepackage{tikz}
\usetikzlibrary{matrix,arrows,decorations.pathmorphing}

\usepackage{hyperref}
\usepackage[capitalize]{cleveref}

\makeatletter
\def\ps@pprintTitle{%
   \let\@oddhead\@empty
   \let\@evenhead\@empty
   \def\@oddfoot{\reset@font\hfil\thepage\hfil}
   \let\@evenfoot\@oddfoot
}
\makeatother
\usepackage{listings}  % listings environment
\usepackage{xargs}     % Use more than one optional parameter in new commands
\usepackage[colorinlistoftodos,prependcaption,textsize=scriptsize]{todonotes}
\newcommandx{\todoblue}[2][1=]{\todo[linecolor=blue,backgroundcolor=blue!25,bordercolor=blue,#1]{#2}}  % Stolen from stackexchange
\newcommandx{\todored}[2][1=]{\todo[linecolor=red,backgroundcolor=red!25,bordercolor=red,#1]{#2}}  % Stolen from you

\newtheorem{theorem}{Theorem}[section]
\newtheorem{lemma}[theorem]{Lemma}

\newtheorem{proposition}[theorem]{Proposition}

\theoremstyle{definition}

\newtheorem{remark}[theorem]{Remark}

\newcommand{\F}{\ensuremath{\mathbb F}}
\newcommand{\FF}{\F}
\newcommand{\Z}{\ensuremath{\mathbb Z}}
\newcommand{\ZZ}{\Z}
\newcommand{\QQ}{\mathbb{Q}}

\newcommand{\HC}{\ensuremath{\mathcal H}}
\let\hc\HC

\newcommand{\spec}{\operatorname{Spec}}
\newcommand{\End}{\operatorname{End}}

\begin{document}

\begin{frontmatter}

\title{Kummer surfaces for primality testing}

\author{Eduardo Ru\'{i}z Duarte \& Marc Paul Noordman}

\address{Universidad Nacional Aut\'{o}noma de M\'{e}xico, University of Groningen \\ \rm\url{rduarte@ciencias.unam.mx}, \url{m.p.noordman@rug.nl}}

\begin{abstract}
We use the arithmetic of the Kummer surface associated to the Jacobian of a hyperelliptic curve to study the primality of integers of the form $4m^2 5^n - 1$.
We provide an algorithm capable of proving the primality or compositeness of most of the integers in these families and discuss in detail the necessary steps to implement this algorithm in a computer. Although an indetermination is possible, in which case another choice of initial parameters should be used, we prove that the probability of reaching this situation is exceedingly low and decreases exponentially with $n$.
\end{abstract}


\end{frontmatter}

%\linenumbers

\newcommand{\sqp}{\mathcal{T}}
\newcommand{\ja}{\mathcal{J}}
\newcommand{\KK}{\mathcal{K}}
\newcommand{\PP}{\mathbb{P}}

\section*{Introduction}

Determining the primality of an arbitrary integer $n$ is a fundamental problem in number theory. The first deterministic polynomial time primality test (AKS) was developed by Agrawal, Kayal and Saxena \cite{agrawalsaxena}. Lenstra and Pomerance in \cite{lenpom} proved that a variant of AKS has running time $(\log n)^6(2+\log\log n)^c$ where $c$ is an effectively computable real number. The AKS algorithm has more theoretical relevance than practical.

If rather than working with general integers, one fixes a sequence of integers, for example Fermat or Mersenne numbers, one can often find an algorithm to determine primality using more efficient methods; for these two examples P\'epin's and Lucas-Lehmer's tests respectively provide fast primality tests. 

Using elliptic curves, in 1985, Wieb Bosma in his Master Thesis \cite{bosma} found analogues of Lucas tests for elements in $\Z[i]$ or $\Z[\zeta]$ (with $\zeta$ a third root of unity) by replacing the arithmetic of $(\Z/n\Z)^\times$ with the arithmetic of elliptic curves modulo $n$ with complex multiplication (CM). Further, Pomerance proved in \cite{pomproof} that for each $p>31$ there is a proof of its primality using a suitable choice of an Elliptic curve $E/\F_p$ and a $\F_p$-rational point $Q\in E(\F_p)$ of order $2^r>(p^{1/4}+1)^2$ using $r$ arithmetic operations in $E$. The difficulty in applying this theorem to the problem of determining primality lies in finding suitable $E$ and $Q$ for a given $p$. Conversely, one can look at certain families of elliptic curves $E_\alpha/\QQ$ equipped with a point $Q_\alpha$ of infinite order and determine which sequences of integers are suitable to establish a primality test using $\langle E_\alpha,Q_\alpha\rangle$. In the case of elliptic curves with complex multiplication, this is worked out in \cite{primeframework}.

Previous methods in the search of general primality tests for any integer $N$ using elliptic curves are Goldwasser-Kilian \cite{gk1,gk2} and Atkin-Morain \cite{am1,morain}. The first generates a random elliptic curve $E$ and a point $P$ modulo $N$, then it uses a counting algorithm (for example Schoof's) to find the number of points of $E$ modulo $N$. If this algorithm finishes unexpectedly then $N$ is composite and a factor can be provided, if it returns then other tests need to be done to check if $\langle E,P\rangle$ suffice to determine the primality of $N$ using the order of $P$ in $E$ modulo $N$. If this does not suffice another elliptic curve should be chosen. The Atkin-Morain's uses essentially the same idea, however instead of relying in a point counting algorithm on a random elliptic curve (which is in practice very slow), it constructs via complex multiplication an elliptic curve where the number of points is easy to compute. Adleman and Huang \cite{Adleman-Huang} 
developed a technique using the Jacobian variety of higher genus curves to improve the heuristics on the success of the previously mentioned Goldwasser-Kilian primality test. 

Interesting primality tests similar to what we present here using Abelian varieties are with curves of genus $0$ and $1$. Hambleton in \cite{hambleton} generalized Lucas-Lehmer primality tests by using the group structure of Pell conics. Using genus $1$ elliptic curves, Gross \cite{grossmersenne} developed a primality test for Mersenne integers. 
Further, Denomme and Savin in \cite{densav} used complex multiplication of an 
elliptic curve to construct primality tests for different families of integers. These primality tests were later generalized using various one-dimensional group schemes by Gurevich and 
Kunyavski\u{\i} \cite{gur}. 

In \cite[Remark 4.13]{primeframework} Abatzoglou, Silverberg, Sutherland and Wong pose the question of whether one can use higher-dimensional abelian varieties to create primality tests. The goal of this paper is to propose such a primality test, based on the arithmetic of the Jacobian of certain genus 2 hyperelliptic curves and their associated Kummer surfaces. Specifically, in this paper, we use the Kummer surface associated to the Jacobian $\ja$ of the hyperelliptic curve $y^2 = x^5 + h$ to study the primality of integers of the form 
\[\lambda_{m,n}:=4 m^2\cdot 5^n-1,\quad\quad m, n \in \Z_{\geq 1}.\]

This paper consists of three sections. In the first, we explain the algorithm from a theoretical point of view. The general idea is that when $\lambda_{m,n}$ is prime, the group of rational points of the Jacobian $\ja/\F_{\lambda_{m,n}}$ of the curve $\HC$ given by $y^2=x^5+h$ is a cyclic $\mathbb{Z}[\sqrt 5]$-module of known order (see Proposition \ref{abelian-group-structure}). We can construct explicitly $[\sqrt{5}]\in \operatorname{End}_{\F_{\lambda_{m,n}}}(\ja)$ (see Section \ref{subsection_sqrt5} and Equations \ref{equations_example_sqrt5} for a full worked example). On the other hand, if $\lambda_{m,n}$ is not prime we can still consider the scheme $\ja/\mathcal{S}$ and construct $[\sqrt{5}]\in\operatorname{End}_{\mathcal{S}}(\ja)$ where $\mathcal{S}:=\spec {\Z/\lambda_{m,n}\Z}$ as we will see in the first section. 
With this, we choose some base point $Q\in\ja$, and study the integer $\text{inf}\{k:[\sqrt{5}]^k P =0\text{ in } \ja(\Z/\lambda_{m,n}\Z)\}$ where $P = 4m^2Q$ to determine the primality or compositeness of $\lambda_{m,n}$. This leads to Theorem \ref{main_theorem}, which is the main theoretical result underlying the algorithm. 

In Section \ref{section_implementation} we make this primality test explicit. The primality test depends on some auxilary data, namely the choice of two integers $\alpha$ and $\beta$ such that $h := \beta^2 -\alpha^5$ is coprime with $\lambda_{m,n}$. This corresponds to the point $Q = (\alpha, \beta)$ on the curve $y^2 = x^5 + h$. We use in this section the Kummer surface $\mathcal{K}$ associated to $\ja$. We will see that $\mathcal{K}$ is a simpler geometrical and arithmetical object compared to $\ja$, which preserves the necessary information to determine compositeness or primality of $\lambda_{m,n}$. We show in this section how to obtain explicit representations for the $[\sqrt 5]$ endomorphism and for the point $P = 4m^2Q$, which are necessary to actually perform the algorithm. After doing these precomputations, the algorithm itself is reasonably straight-forward, see \cref{primality_pseudocode}. 

The algorithm has an indeterminate case, corresponding to case 3 in \cref{main_theorem}. If this happens, one has to change the auxiliary data $(\alpha, \beta)$ and run the test again. This means recomputing the representations of $[\sqrt 5]$ and $P_0$. This is an expensive computation. However, in \cref{probability-section} we show that the probability of this happening is vanishingly small for even moderately sized $m$ and $n$, at least under the assumption that $\lambda_{m,n}$ is prime. Specifically, we show that if $\lambda_{m,n}$ is prime and at least 100, and if the coordinates $(\alpha, \beta)$ of $Q$ are chosen randomly from all integers between $0$ and $\lambda_{m,n}$ for which the integer $h = \beta^2 - \alpha^5$ is not a multiple of $\lambda_{m,n}$, then the probability of ending up in the indeterminate case is less than $2m \cdot 5^{-n/2}$. Thus in practise, for large $n$ the algorithm essentially always proves primality or compositeness without need to chance the auxiliary data.

\section*{Acknowledgements}
We thank Prof. Michael Stoll for his valuable comments, ideas and examples with MAGMA regarding the explicit calculation of $[\sqrt{5}]$ on early stages of this paper.

\section{Theory}

In this section, we fix numbers $n, m \in \ZZ_{\geq 1}$ and we consider the number 
\[ \lambda_{m,n} = 4m^2 5^n - 1. \]
If $n$ is even, then {$\lambda_{m,n} = (2m 5^{n/2} - 1)(2m 5^{n/2}+1)$} is composite, so we will always assume that $n$ is odd. Also, any factors $5$ in $m$ can be absorbed in $n$, so we will also assume that $5 \nmid m$. 

We also fix some non-zero integer $h$ which is coprime with $\lambda_{m,n}$, and we consider the hyperelliptic curve $\mathcal H$ defined by $y^2 = x^5 + h$. Our primality test uses that the Jacobian of this curve admits real multiplication by $\sqrt{5}$. We often need to consider the reduction of $\mathcal H$ and its Jacobian modulo $\lambda_{m,n}$, even in cases where this number might be composite, and so we use the language of schemes. Therefore, in the rest of this section, we always view $\mathcal H$ as a projective curve (or more precisely, an arithmetic surface) over the scheme $S := \spec \ZZ[\frac{1}{10h}, \sqrt{5}]$. Over this curve, $\mathcal J$ is smooth with complete, geometrically connected curves as fibers, and so its relative Jacobian $\mathcal J := \operatorname{Jac}(\mathcal H/S)$ is well-defined. The following proposition shows that $\End_S(\mathcal J) = \ZZ\left[\frac{1 + \sqrt{5}}{2}\right]$, and in particular we have a $[\sqrt 5]$-map on $\mathcal{J}$ defined over $S$. 

\begin{proposition}\label{endomorphisms}
Consider $\omega = \frac{\textrm{d} x}{y}$ as a differential on $\mathcal J$. There is a unique ring isomorphism $\ZZ\left[\frac{1 + \sqrt{5}}{2}\right] \overset{\sim}{\longrightarrow}\End_S(\mathcal J), \,\,\alpha \mapsto [\alpha]$ such that $[\alpha]^*\omega = \alpha \omega$ for all $\alpha \in \ZZ\left[\frac{1 + \sqrt{5}}{2}\right]$. 
\end{proposition}
\begin{proof}
In \cite{PROLEGOMENA}, Chapter 15, it is proven that the Jacobian of $\mathcal H_{\overline \QQ}$ is simple, and has endomorphism ring $\ZZ[\zeta]$, where $\zeta \in \overline \QQ$ is a fifth root of unity. The action of $\zeta$ on $\mathcal{J}$ is induced by the map $\mathcal H_{\overline \QQ} \to \mathcal H_{\overline \QQ}$ sending $(x,y)$ to $(\zeta x, y)$. From this description, one easily computes that $[\zeta]^*\omega = \zeta\omega$, and so the same holds for all $\alpha \in \ZZ[\zeta]$. We have a canonical morphism of rings $\phi\colon \End_S(\mathcal J) \to \End_{\overline \QQ}(\mathcal J_{\overline \QQ}) = \ZZ[\zeta]$, and this morphism is injective. Thus we need to show that $\phi$ has image $\ZZ\left[\frac{1+\sqrt{5}}{2}\right]$. 

Let $R$ be the image of $\phi$. First we show that $R \subset \ZZ\left[\frac{1+\sqrt{5}}{2}\right]$. Indeed, the above morphism $\phi$ sends an endomorphism $E$ of $\mathcal J$ to the eigenvalue of $\omega$ under the action of $E$ on the cotangent space at the origin. If $E$ is defined over $S$, then so are both $\omega$ and $E^*\omega$, and so the eigenvalue of $\omega$ is defined over $S$ as well. Thus, we get $\phi(E) \in \ZZ[\frac{1}{10h},\sqrt 5] \cap \ZZ[\zeta] = \ZZ\left[\frac{1+\sqrt{5}}{2}\right]$ as claimed. 

Now we show that $\ZZ\left[\frac{1+\sqrt{5}}{2}\right] \subset R$. For this, let $S'= \spec \Z[\frac{1}{10h}, \zeta]$. The canonical map $S'\to S$ is an (unramified) \'etale covering of degree 2, hence Galois, and the Galois action is given by $\zeta \mapsto \zeta^4$. Note that $\End_{S'}(\mathcal J_{S'}) = \ZZ[\zeta]$, since the endomorphism $\zeta$ is defined over $S'$. Since we have that 
\[ \frac{1 + \sqrt 5}{2} = \pm(\zeta + \zeta^4) + 1,\]
(the sign depends on the choice of $\zeta$) and the right hand side is clearly Galois invariant, the theory of Galois descent for endomorphisms on abelian varieties shows that the endomorphism of $\mathcal J_{S'}$ corresponding to $\frac{1 + \sqrt 5}{2}$ descends to an endomorphism of $\mathcal J$ over $S$. Hence $\frac{1 + \sqrt 5}{2} \in R$, and we are done. \end{proof}

Note that over $\ZZ[\frac{1}{10h}, \sqrt{5}]$ we may factor $\lambda_{m,n}$ as 
\[ \lambda_{m,n} = (2m \sqrt{5}^n + 1)(2m\sqrt{5}^n - 1). \]
These two factors are coprime over $\ZZ[\frac{1}{10h}, \sqrt{5}]$ because their difference is a unit. 
\begin{lemma}\label{modncalculations}
Assume that $h$ and $\lambda_{m,n}$ are coprime. The canonical map
\[ \frac{\ZZ}{\lambda_{m,n}\ZZ}\, \longrightarrow\, \frac{\ZZ[\frac{1}{10h}, \sqrt{5}]}{(2m\sqrt{5}^n - 1)} \]
is an isomorphism of rings, and its inverse is given by the map
\[ \frac{\ZZ[\frac{1}{10h}, \sqrt{5}]}{(2m\sqrt{5}^n - 1)} \longrightarrow  \frac{\ZZ}{\lambda_{m,n}\ZZ} \quad\quad \sqrt{5} \mapsto 2m\cdot 5^{(n+1)/2}, \,\,\frac{1}{10h} \mapsto m^2\cdot 5^{n-1}h^{-1}\]
\end{lemma}
\begin{proof}
Easy computation. 
\end{proof}

\begin{remark}
In what follows we will always assume that $\lambda_{m,n}$ is coprime with $h$, and identify $\ZZ/\lambda_{m,n} \ZZ$ and $\ZZ[\frac{1}{10h}, \sqrt{5}]/(2m\sqrt{5}^n - 1)$. In practice this just means that we have selected a `canonical' square root of $5$ in $\ZZ/\lambda_{m,n} \ZZ$, namely $2m\cdot 5^{(n+1)/2}$. In particular, the base changes of $\mathcal H$ and $\mathcal J$ to $\ZZ/\lambda_{m,n} \ZZ$ are well-defined: formally, they are the base change of $\mathcal H$ and $\mathcal J$ via the map $\spec (\ZZ/\lambda_{m,n}\ZZ) \to S$ corresponding to the ideal of $\ZZ[\frac{1}{10h}, \sqrt{5}]$ generated by $2m\sqrt{5}^n - 1$. We will denote these base changes by $\mathcal H_{\lambda_{m,n}}$ and $\mathcal J_{\lambda_{m,n}}$. Note that this base change depends on a choice, as we could as well have chosen the ideal corresponding to $2m\sqrt{5}^n + 1$. 

Base change gives a canonical map $\ZZ\left[\frac{1 + \sqrt 5}{2}\right] \to \End_{\ZZ/\lambda_{m,n}\ZZ}(\mathcal J_{\lambda_{m,n}})$. In other words, for any $\alpha \in \ZZ\left[\frac{1 + \sqrt 5}{2}\right]$ we have a canonical endomorphism $[\alpha]$ of $\mathcal J_{\lambda_{m,n}}$ defined over $\ZZ/\lambda_{m,n}\ZZ$. In what follows, the endomorphism $[\sqrt 5]$ will play an important role, and the main consequence of the results above is that we can make a \emph{consistent} choice of these endomorphisms, defined over $\ZZ/\lambda_{m,n}\ZZ$, for each choice of $m$ and $n$.
\end{remark}

We will now study the structure of the group $\mathcal J(\ZZ/\lambda_{m,n}\ZZ)$ in the case where $\lambda_{m,n}$ is prime. We start with the $2$-torsion. 
\begin{proposition}\label{2torsionforprimality}
Suppose $\lambda_{m,n}$ is prime. Then $\ja[2](\F_{\lambda_{m,n}})\cong \Z/(2)\times\Z/(2)$. 
\end{proposition}
\begin{proof}
We know that $\ja[2](\F_{\lambda_{m,n}})\subset\ja(\F_{\lambda_{m,n}})$ consists of divisor classes $D-2\infty$ where $D$ 
is a the sum of a pair of distinct Weierstrass points of $\hc$ and $D$
is fixed under the action of the absolute Galois group of $\F_{\lambda_{m,n}}$. 
Since $\gcd(\lambda_{m,n}-1, 5) = 1$, there is a unique $\alpha \in \F_{\lambda_{m,n}}$ with $\alpha^5 = -h$. Then the Weierstrass points of $\hc$ are the point $\infty$ at infinity and the points of the form $(\zeta^j\alpha,0)$ for $0\leq j\leq 4$. 
Exactly two of these Weierstrass points are defined over $\F_{\lambda_{m,n}}$, namely $(\alpha,0)$ and $\infty$. 
The other four are defined over the quadratic extension of $\F_{\lambda_{m,n}}$, because 
$\zeta$ lies there. Since $\lambda_{m,n} \equiv 4 \mod 5$, $\zeta$ and $\zeta^4$ are Galois conjugate, as are $\zeta^2$ and $\zeta^3$. Hence, the Galois action fixes the points $\infty$ and $(\alpha, 0)$, it interchanges the pair $(\zeta \alpha, 0)$ and $(\zeta^4 \alpha, 0)$, and it interchanges the pair $(\zeta^2\alpha, 0)$ and $(\zeta^3\alpha, 0)$. Therefore there are exactly three unordered pairs of Weierstrass points stable under the Galois action, namely 
\[\{(\alpha, 0), \infty\}, \quad \{(\zeta\alpha, 0), (\zeta^4\alpha , 0)\} \quad \textrm{and}\quad \{(\zeta^2\alpha, 0), (\zeta^3\alpha, 0)\}.\]
This means that the group $\ja(\F_{\lambda_{m,n}})$ has exactly three points of order 2. Together with the identity element, this shows that $\#\ja[2](\lambda_{m,n}) = 4$, and so $\ja[2](\F_{\lambda_{m,n}})\cong \Z/(2)\times\Z/(2)$. 
\end{proof}

Observe that we can obtain explicitly the $\F_{\lambda_{m,n}}$-rational zero $\alpha$ of $x^5 + h\in\F_{\lambda_{m,n}}[x]$ as follows.
We know that there is a $d\in \Z$ such that the map $x\mapsto (x^d)^{5}$ defined over $\F_{\lambda_{m,n}}$ is the identity map. By Fermat's little theorem, $d$ satisfies $5d\equiv 1\bmod (\lambda_{m,n}-1)$ and $\lambda_{m,n}-1=4m^2\cdot 5^n -2$. To calculate $d$, let $N=2m^2\cdot 5^n -1$ and write $\lambda_{m,n}-1=2N$. Using the Chinese Remainder Theorem we evaluate $5^{-1}$
with the isomorphism $\tau:\Z/2N\Z\to \Z/2\Z\times \Z/N\Z$, using the fact that $5^{-1}\equiv 2m^2\cdot 5^{n-1}\bmod N$ and $5$ is odd. Hence, $\tau(5^{-1})=(1,2m^2\cdot 5^{n-1})=(1,0)+(0,2m^2\cdot 5^{n-1})$ and therefore: 
\[
d=5^{-1}=\tau^{-1}(1,0)+\tau^{-1}(0,2m^2\cdot 5^{n-1})=N+2m^2\cdot 5^{n-1}=12m^2\cdot 5^{n-1}-1.
\]
Using this we have that $x^{5d}=x$ in $\F_{\lambda_{m,n}}$, and particularly if $x=-h$, we obtain:
\begin{equation}\label{rootxh}
\alpha=(-h)^{d}=(-h)^{12m^2\cdot 5^{n-1}-1}.
\end{equation}

Proposition \ref{2torsionforprimality} allow us to deduce the full group structure of $\ja(\FF_{\lambda_{m,n}})$. For an abelian group $G$ and a prime $p$, we will denote by $G[p^\infty]$ its subgroup of elements whose order is a power of $p$. \\

\begin{proposition}\label{abelian-group-structure}
Assume that $\lambda_{m,n}$ is prime. Then we have 
\[ \mathcal J(\FF_{\lambda_{m,n}}) \cong \big( \ZZ/4m^25^n\ZZ\big)^2 \]
as abelian groups. 
\end{proposition}
\begin{proof}
For the proof, we will write $\mathcal H$, $\ja$ and $\FF$ for $\mathcal{H}_{\lambda_{m,n}}$, $\mathcal{J}_{\lambda_{m,n}}$ and $\FF_{\lambda_{m,n}}$ respectively. First we calculate the zeta function of $\mathcal H$. We refer to a paper by Tate and Shafarevich \cite{cardsuper} in which they give an explicit description for the numerator of the zeta function of the curve $\mathcal{C}/\F_p$ given by $y^e=x^f + \delta$ in the case $\mu=\text{lcm}(e,f)|p^k+1$ for some $k$. In our case $p=\lambda_{m,n}=4m^2\cdot 5^n -1$, $\mu=10$ and $k=1$. 
By \cite{cardsuper} the numerator of the zeta function of $\mathcal H/\FF$ is in this case given by $\lambda_{m,n}^2 T^4 + 2\lambda_{m,n} T^2 + 1$, which tells us the characteristic polynomial $\chi_\ja(T)$ of Frobenius of $\ja$ equals $ T^4 + 2\lambda_{m,n} T + \lambda_{m,n}^2=(T^2 + \lambda_{m,n})^2$. With this information, we obtain 
\[\#\ja(\FF)=\chi_\ja(1)=16 m^4\cdot  5^{2n}.\]
A finite abelian group is the product of its Sylow subgroups for all primes dividing the order of the group. Therefore, to show that $\ja(\FF)$ and $(\ZZ/4m^25^n\ZZ)^2$ are isomorphic as groups, it is sufficient to show that they have the same $p$-Sylow group for all primes. We first look at the odd primes. Let $p$ be odd, and let $k$ be the $p$-adic valuation of $\#\ja(\FF)$ (i.e. the integer $k$ such that $p^k$ divides $\#\ja(\FF)$ but $p^{k+1}$ does not). Since $\#\ja(\FF) = 16m^25^{2n}$ is a square, $k$ is even. Lemma 3.1 of \cite{onsuper}, together with the factorization $\chi_\ja(T) = (T^2 + \lambda_{m,n})^2$, tells us that 
\[ \ja(\FF)[p^\infty] = (\ZZ/p^{k/2}\ZZ)^2 \]
as expected. 

The case $p = 2$ requires more work. First, \cref{2torsionforprimality} shows that $\ja(\FF)[2] = (\ZZ/2\ZZ)^2$. Therefore, we get that 
\[ \ja(\FF)[2^\infty] \cong (\ZZ/2^a\ZZ) \times (\ZZ/2^b\ZZ) \]
for certain $a, b \geq 1$. Without loss of generality, we may assume that $a \geq b$. We want to prove that $a = b$. Suppose towards a contradiction that $a > b$. Let 
\[\phi\colon \ja(\FF)[2^\infty] \to (\ZZ/2^a\ZZ)\times (\ZZ/2^b\ZZ)\]
be an isomorphism of groups. Let $P \in \ja(\FF)[2^\infty]$ be an arbitrary element of order $2^a$. Then $\phi(P)$ is of the form $(s,t)$ with $s \in \ZZ/2^a\ZZ$ of order $2^a$ and $t \in \ZZ/2^b\ZZ$ arbitrary. Then $2^{a-1}s \in \ZZ/2^a\ZZ$ has order 2, and is therefore equal to $2^{a-1}$. And $2^{a-1}t = 0$, since $a-1 \geq b$. Hence, we see that $\phi(2^{a-1}P) = (2^{a-1}s, 2^{a-1}t) = (2^{a-1}, 0)$, which is independent of the choice of $P$. Since $\phi$ is an isomorphism, we conclude that $2^{a-1}P = 2^{a-1}Q$ for any two points $P$ and $Q$ in $\ja(\FF)$ of order $2^a$. 

Now consider the endomorphism $\theta = [\frac{1 + \sqrt 5}{2}] \in \End_{\FF}(\ja)$ that we know exists by \cref{endomorphisms}. A short computation shows that $\theta \circ (\theta - 1) = \textrm{id}_\ja$. In particular $\theta$ and $\theta - 1$ are both automorphisms. This implies that $\theta$ preserves the order of elements of $\ja$, and also that $\theta$ has no non-trivial fixed points (because otherwise $\theta - 1$ would not be injective). Now let $P \in \ja(\FF)$ be a point of order $2^a$. Then also $\theta(P)$ is a point of $\ja(\FF)$ of order $2^a$. By the above independence, we now have that 
\[ 2^{a-1}P = 2^{a-1}\theta(P) = \theta(2^{a-1}P).\]
But $2^{a-1}P \neq 0$, so $2^{a-1}P$ is a non-trivial fixed point of $\theta$, which is not possible. This contradiction shows that $a > b$ is not possible. We conclude that $a = b$, so that  
\[ \ja(\FF)[2^\infty] \cong (\ZZ/2^a\ZZ)^2.\]
The result follows.
\end{proof}

We will also need to understand the action of $[\sqrt 5]$ on the $5$-power torsion. 
\begin{proposition}\label{sqrt5-module}
Assume that $\lambda_{m,n}$ is prime and that $n > 1$. Consider $\mathcal J(\FF_{\lambda_{m,n}})[5^\infty]$ as a $\ZZ[\sqrt 5]$-module via the map $[\sqrt 5]$ as above. Then we have 
\[ \mathcal J(\FF_{\lambda_{m,n}})[5^\infty] = 4m^2\cdot \mathcal J (\FF_{\lambda_{m,n}})\cong \ZZ[\sqrt 5]/(\sqrt 5^{2n}) \]
as $\Z[\sqrt 5]$-modules.
\end{proposition}
\begin{proof}
For the proof, we will write $\mathcal H$, $\ja$ and $\FF$ for $\mathcal{H}_{\lambda_{m,n}}$, $\mathcal{J}_{\lambda_{m,n}}$ and $\FF_{\lambda_{m,n}}$ respectively. 

The first claimed equality follows directly from \cref{abelian-group-structure}. Moreover, that same proposition tells us that 
\[ \ja(\FF)[5^\infty] = (\ZZ/5^n\ZZ)^2 \]
as abelian groups. In particular, the endomorphism $[5^n] = [\sqrt 5]^{2n}$ acts as the zero map on $\ja(\FF)[5^\infty]$, so we may view $\ja(\FF)$ as a module over $\ZZ[\sqrt 5]/(\sqrt 5^{2n})$. This ring is an artinian principal ideal ring (its ideals are of the form $(\sqrt 5^k)$ for $k = 0,\ldots, 2n$), and the structure theorem for modules over such rings shows that any finitely generated module over such a ring is is a product of cyclic modules. Hence, we may write 
\[ \ja(\FF)[5^\infty] = \prod_{i = 1}^r \frac{\ZZ[\sqrt 5]}{(\sqrt 5^{e_i})} \]
for certain integers $e_1 \geq \cdots \geq e_r \geq 1$. To get the number of elements correct, we need that $e_1 + \ldots + e_r = 2n$. Since $[5^{n-1}] = [\sqrt 5]^{2n-2}$ does not act as the zero map on $\ja(\FF)[5^\infty]$, we need that $e_1 \geq 2n-1$. Since $n > 1$, we have $e_1 \geq 3$, and we see that the factor $\ZZ[\sqrt 5]/(\sqrt 5^{e_1})$ contains exactly 24 elements of order 5. But from the structure of $\ja(\FF)[5^\infty]$ as abelian group, we know that in total it contains 24 elements of order 5. Hence we have $r = 1$, and so $e_1 = 2n$, and the result follows. 
\end{proof}

We now arrive at the main theorem of this section. 

\begin{theorem}\label{main_theorem}
Let $n,m \in \ZZ$ with $n$ odd and $5 \nmid m$ s.t $m^2<\tfrac{(\sqrt{5}^n - 1)^4 + 1}{4\cdot 5^n}$.
Set as before $\lambda_{m,n} = 4m^2 5^n - 1$, and assume that $\gcd(\lambda_{m,n}, h) = 1$. Let $Q \in \mathcal J(\ZZ/\lambda_{m,n} \ZZ)$ be any point, and define $P = 4m^2 \cdot Q$. 
Let 
\[ r = \inf \big\{k \,: \,[\sqrt 5]^k P = 0 \textrm{ in } \ja(\ZZ/\lambda_{m,n}\ZZ)\big\} \,\in \mathbb{N} \cup \{\infty\}. \]
\begin{enumerate}
    \item If $r > 2n$, then $\lambda_{m,n}$ is composite.
    \item If $4\cdot\log_5(\sqrt[4]{\lambda_{m,n}} + 1) < r \leq 2n$, then either
     \begin{itemize}
        \item[\arabic{enumi}a)]  $\lambda_{m,n}$ is prime, or
        \item[\arabic{enumi}b)]  there is a prime $p | \lambda_{m,n}$ such that $[\sqrt5]^{r-1} P = 0$ mod $p$.  
    \end{itemize}
    \item If $r \leq 4\cdot\log_5(\sqrt[4]{\lambda_{m,n}} + 1)$, then either
    \begin{itemize}
        \item[\arabic{enumi}a)] $\lambda_{m,n}$ is composite, or
        \item[\arabic{enumi}b)] $\lambda_{m,n}$ is prime and there exists a point $Q' \in \ja(\ZZ/\lambda_{m,n}\ZZ)$ with $Q = [\sqrt 5]^{2n-r}(Q')$. 
    \end{itemize} 
\end{enumerate}
\end{theorem}
\begin{proof}
\begin{enumerate}
    \item If $\lambda_{m,n}$ is prime, then by \cref{sqrt5-module}, we have $Q = 4m^2 P \in \ja(\FF_{\lambda_{m,n}})[5^\infty]$. By that same proposition, this group is annihilated by $[\sqrt 5]^{2n}$. Hence, if $\lambda_{m,n}$ is prime then $[\sqrt 5]^{2n} P = 0$ and so $r \leq 2n$. 
    \item Let $p$ be the smallest prime dividing $\lambda_{m,n}$, and assume that $[\sqrt 5]^{r-1} P \neq 0$ in $\ja(\FF_p)$. We show that $\lambda_{m,n} = p$. Indeed, the group $\ja(\FF_p)$ is a $\ZZ[\sqrt 5]$-module via the action of $[\sqrt 5]$, and the assumption that $[\sqrt 5]^{r-1} P \neq 0$ while $[\sqrt 5]^{r} P = 0$ in $\ja(\FF_p)$ implies that $P$ generates a $\ZZ[\sqrt 5]$-submodule of $\ja(\FF_p)$ isomorphic to $\ZZ[\sqrt 5]/(\sqrt 5^{r})$. In particular we have that 
    \[\#\ja(\FF_p) \geq 5^{r} > (\sqrt[4]{\lambda_{m,n}} + 1)^4.\] 
    On the other hand the Hasse-Weil inequality gives
    \[ \#\ja(\F_{p}) \leq (\sqrt{p} + 1)^4. \] 
    Comparing these inequalities, we find that $p > \sqrt{\lambda_{m,n}}$. This is possible only if $p = \lambda_{m,n}$, and in particular $\lambda_{m,n}$ is prime. 
    \item Suppose $\lambda_{m,n}$ is prime. By \cref{sqrt5-module}, $\ja(\FF_{\lambda_{m,n}}) = \ZZ[\sqrt 5]/(\sqrt 5^{2n})$ as $\ZZ[\sqrt 5]$-modules. Since $[\sqrt 5]^{r} P = 0$ and $[\sqrt 5]^{r-1} P \neq 0$, there is a point $P' \in \ja(\FF_{\lambda_{m,n}})$ with $P = [\sqrt 5]^{2n-r}P'$. Choose integers $a$ and $b$ such that $4m^2a + 5^nb = 1$ (such integers exist because $4m^2$ and $5^n$ are coprime). Then define $Q' = aP' + [\sqrt 5]^{r}(bQ) \in \ja(\FF_{\lambda_{m,n}})$. Then we have
        \[    [\sqrt 5]^{2n-r}Q' = [\sqrt 5]^{2n-r}aP' + 5^n b Q = aP + 5^nbQ = (4m^2a + 5^nb)Q = Q\]
    as needed. \qedhere
    \end{enumerate}
\end{proof}

\section{Implementation}\label{section_implementation}

In this section, we describe how to implement \cref{main_theorem} as an algorithm to test primality of numbers of the form $\lambda_{m,n} = 4m^2 \cdot 5^n - 1$. The algorithm depends on the auxiliary data of the hyperelliptic curve $\mathcal{H}\colon y^2 = x^5 + h$ with $h \in \mathbb{Z}$ and a base point $Q_0 = (\alpha, \beta) \in \mathcal{H}(\QQ)$ whose image in $\ja$ has infinite order (e.g.,  $h= -\alpha^5 + \beta^2$ with $\alpha,\beta\in\mathbb{Z}$, $h\nmid\lambda_{m,n}$ and $[(\alpha,\beta)-\infty]\in\ja(\QQ)$ of infinite order).

This algorithm consists of three parts. First one has to compute an explicit expression for the $[\sqrt{5}]$ morphism of the Jacobian $\mathcal{J}$ of $\mathcal{H}$. Secondly, one has to compute the expression $P_0 = 4m^2\cdot Q_0 \in \mathcal{J}(\mathbb{Q})$. Finally, one has to apply $[\sqrt 5]$ iteratively on $P_0$ and compare the result to the cases in \cref{main_theorem}. Note that the first step does not depend on $m$ and $n$, and the second step does not depend on $n$. Hence, for a fixed choice of $m$ and a fixed choice of the auxiliary data, one only has to perform steps 1 and 2 once, and the output of these steps can then be used to test primality of $\lambda_{m,n}$ for any value of $n$. This is important, because the first two steps are reasonably time and resource intensive, and require computations in the Jacobian of $\mathcal{H}$ which requires specialized mathematical software like MAGMA. The third step, on the other hand, consists of applying explicit polynomials repeatedly to an explicit vector of numbers, and therefore can be done in general purpose programming languages like Python. 

As mentioned in the introduction, the algorithm has an indeterminate case, corresponding to case 3 in \cref{main_theorem}. If this happens, one has to change the auxiliary data and run the test again. However, in \cref{probability-section} we show that the probability of this happening is vanishingly small for even moderately sized $m$ and $n$, at least under the assumption that $\lambda$ is prime. 

\subsection{Using the Kummer surface \texorpdfstring{$\mathcal{K}$}{K} of \texorpdfstring{$\ja$}{J} instead of \texorpdfstring{$\ja$}{J}}
In order to do explicit computations with elements of the Jacobian, one needs a way of representing the elements of $\mathcal J$. One way of doing this would be to embed $\mathcal{J}$ in projective space. The Jacobian of the curve $\mathcal{H}$ embeds into $\mathbb{P}^8$, see \cite{GRANTFORMAL} for explicit formulas (if $\mathcal{H}$ did not have a rational Weierstrass point, one would even need $\mathbb{P}^{15}$, see \cite{GLJCG2}). Unfortunately, this large number of coordinates turns out to be impractical computationally. Another option is to use the fact that elements of $\mathcal{J}(\overline \QQ)$ are represented by divisors of the form $P_1 + P_2 - 2\cdot\infty$ for some $P_1, P_2 \in \mathcal{H}(\overline\QQ)$, for example using \emph{Mumford coordinates}, see \cite{GLOJ}. However, when using these coordinates one often has to distinguish between divisors based on whether 0, 1 or 2 of the points $P_1$ and $P_2$ are equal to $\infty$. In our case, this leads to complicated formulas involving these different cases. 

To avoid these difficulties, we work not with points on the Jacobian $\mathcal{J}$, but with points on the associated \emph{Kummer surface} $\mathcal{K} = \mathcal{J}/\langle[\pm 1]\rangle$ given by modding out the involution on $\mathcal{J}$. This object is not an algebraic group anymore, because addition is not well-defined. However, each endomorphism of $\mathcal{J}$ as abelian variety descends to an endomorphism of $\mathcal{K}$, because every endomorphism of $\mathcal{J}$ commutes with the $[-1]$ map (e.g point doubling in $\mathcal{K}$ \cite[Chapter 3, Theorem 3.4.1, example 3.6.2]{PROLEGOMENA}). In particular, there is still a $[\sqrt 5]$ map $\mathcal{K} \to \mathcal{K}$. Moreover, $\mathcal{K}$ embeds as a quartic surface in $\mathbb{P}^3$ (see \cite[Chapter 3, Equation 3.1.8]{PROLEGOMENA}), so we can represent points on $\mathcal{K}$ with four coordinates. A nice additional benefit of using the Kummer surface is that the $[\sqrt 5]$ endomorphism on $\mathcal{K}$ is defined already over $\mathbb{Q}$, rather than over $\mathbb{Q}(\sqrt 5)$, so the formulas we find involve only rational numbers. See \cite[Chapter 3]{PROLEGOMENA} and \cite[Section 5]{reza} for more background on the Kummer surface and its embedding into $\mathbb{P}^3$. For the rest of the section, we fix the quotient map $\kappa\colon \mathcal{J} \to \mathcal{K}$ and the embedding $\iota\colon \KK \to \PP^3$ as defined in \cite[Section 2]{FLYNNJAC}; these maps are implemented in MAGMA.

\subsection{Computation of \texorpdfstring{$[\sqrt 5]$}{[sqrt5]}}\label{subsection_sqrt5}

Because we consider $\KK$ as embedded in $\PP^3$, the morphism $[\sqrt 5]\colon \KK \to \KK$ can be written in the form
\begin{align}\label{kumendo}
\begin{split}
\widehat\varphi:\mathcal{K}\subset\mathbb{P}^3&\to\mathcal{K}\subset\mathbb{P}^3 \\
P:=[x_0:x_1:x_2:x_3]&\mapsto [\widehat\varphi_0(P):\widehat\varphi_1(P):\widehat\varphi_2(P):\widehat\varphi_3(P)]
\end{split}
\end{align}
where the $\hat\varphi_i$ are homogeneous polynomials of some degree $N \geq 1$ (it seems that in our case, we can always take $N = 5$). Note that these polynomials are not uniquely determined, for two reasons: one can multiply the four polynomials by a constant, and one can add to each $\hat \varphi_i$ an arbitrary homogeneous polynomial of degree $N$ that vanishes identically on $\KK$. 

To determine explicit polynomials $\hat\varphi_i$, we use an interpolation strategy. That is, we first generate a large number of pairs of points $(P,Q)$ with $P, Q \in \KK$ such that $Q = [\sqrt 5]P$, and then we solve a linear system of equations to obtain the coefficients of the $\hat \varphi_i$. Roughly, the steps are the following. 
\begin{enumerate}
    \item Generate a sufficiently large set $S\subset\ja(\mathbb{Q}(\zeta))$ using $Q_0$ and the action of $\zeta$ on $\ja$. For example $S=\{[a + b\zeta + c\zeta^2 + d\zeta^3] Q_0\}$ for $a,b,c,d \in \{-B, \ldots, B\}$ for some sufficiently large integer $B$. If $P$ and $-P$ are both in $S$, drop one of them. 
    \item Calculate the pairs $(P, [\sqrt 5]P) \in \mathcal{J}(\mathbb{Q}(\zeta))^2$ for each $P \in S$. 
    \item Calculate the set $T = \{(\kappa(P),\kappa([\sqrt 5](P))) : P \in S\}$, as a subset of $(\mathbb{P}^3)^2$.  
    \item Construct a projective system of linear equations $\mathcal{L}$ using the set $T$. Use this system to deduce the coefficients of four homogeneous polynomials of degree $N$ that express $\hat\varphi:\mathcal{K}\to\mathcal{K}$ in $\mathbb{P}^3$. (If $N$ is unknown, simply choose large enough $N$.) 
    \item Remove any common factors in the $\hat \varphi_i$ (in case $N$ was larger than needed).
    \item Check the validity of the $\hat\varphi_i$ with a generic point computation, which uses the quartic equations defining $\mathcal{K}$. 
\end{enumerate} 
We now explain these steps in more detail. 

For step 1, we use the implementation of Jacobians in MAGMA \cite{Magma}. In this computer algebra system, points on the Jacobian are represented in Mumford coordinates. The idea of this representation is to encode the divisor class $[(x_1,y_1)+(x_2,y_2)-2\infty]$ as a pair of polynomials $\langle u(X),v(X)\rangle$ such that $u(x_i)=0$, $v(x_i)=y_i$, $\deg u\leq 2$ and $\deg v\leq 1$. For generic choices of $(x_i, y_i)$, these are given explicitly by
\begin{align}\label{MumfordGeneric}
\begin{split}
\langle u(X),v(X)\rangle&:=\langle X^2 -AX+B,CX+D\rangle\\ 
&=\langle X^2 -(x_1+x_2)X+x_1x_2,\tfrac{y_1-y_2}{x_1-x_2}X+\tfrac{x_2y_1-x_1y_2}{x_1-x_2}\rangle. 
\end{split}
\end{align}
The other \textit{non-generic} cases $[(x_1,y_1)-\infty]$, $[2(x_1,y_1) - 2\infty]$ and $[0]$ are represented respectively by $\langle X-x_1,y_1\rangle$, $\langle X^2-(x_1+x_2)X+x_1x_2,\tfrac{f'(x_1)}{2y_1}X-\tfrac{f'(x_1)}{2y_2}x_1+y_1\rangle$ and $\langle 1,0\rangle$ where $Y^2=f(X)$ is the equation of the hyperelliptic curve $\HC$ associated to $\ja$. The addition of points on $\ja$ 
%\todo{Is there a reason we use $\oplus$ in this paragraph instead of $+$ (which we use in the rest of the paper)? -MP}\todoblue{We can switch to + but I am not sure where in the rest of the paper we use +, i thought to use this in order to differentiate from the "formal sum", but maybe does not matter a lot, what do you think?}\todo{OK if you are fine with it, I switched to +. You are right by the way that we do not use addition as much as I thought. Still I think + looks cleaner than $\oplus$. -MP} 
is already implemented in MAGMA, so we only need to implement the action of $[\zeta^m]$ on $\ja$. It is easy to see from the application of the action on each point in the support of any divisor in $\ja$, that for Mumford representation $[\zeta^m]$ acts as:
\begin{align}\label{zetaaction}
\begin{split}
[\zeta^m]\langle X^2 -AX+B,CX+D\rangle &=\langle X^2 -\zeta^m AX+\zeta^{2m}B,\zeta^{-m}CX+D\rangle.
\end{split}
\end{align}
Using this, one can compute the action of $[\alpha]$ on $\ja$ for any $\alpha \in \mathbb{Z}[\zeta]$. For example, since $\sqrt 5 = 1 + 2\zeta^2 + 2\zeta^3$, we have that
\begin{equation}\label{sq5}
[\sqrt{5}]\langle u(X),v(X)\rangle=\langle u(X),v(X)\rangle + 2\zeta^2\langle u(X),v(X)\rangle + 2\zeta^3 \langle u(X),v(X)\rangle.
\end{equation}

Having implemented the action of $\mathbb{Z}[\zeta]$ on $\mathcal J$, it is easy to compute the set $S$ in step 1. We start with the point $Q_0 = (\alpha, \beta) \in \mathcal H$, which we identify as usual with the point $[(\alpha, \beta) - \infty]$ on the Jacobian, and we compute the points $[a + b\zeta + c\zeta^2 + d\zeta^3] Q_0 \in \mathcal{J}(\mathbb{Q}(\zeta))$ for $a, b, c, d \in \{-B, \ldots, B\}$ for some integer $B$ (in our implementation, $B = 4$ has always been sufficient). Since we are interested in points on $\KK$, for each pair $P$ and $-P$ in $S$ we remove one of them. Of course, one could speed this up by avoiding these double computations from the start, by only computing $[a + b\zeta + c\zeta^2 + d\zeta^3] Q_0$ for tuples $(a,b,c,d)$ with the first non-zero coordinate positive. 

For step 2, we use the action of $[\sqrt 5]$ as described in Equation \ref{zetaaction} to compute $[\sqrt 5]P$ for each $P \in S$. 

In step 3, we use MAGMA's implementation of the map $\kappa\colon \mathcal{J} \to \KK \subset \mathbb{P}^3$ to compute the pairs $(\kappa(P), \kappa([\sqrt 5]P))$ for each $P \in S$. Explicitly, this gives us a large collection $T$ of pairs $(\mathbf{v}, \mathbf{w})$, where $\mathbf{v}$ and $\mathbf{w}$ are vectors of length four and coefficients in $\mathbb{Q}(\zeta)$. Each $\mathbf{v}$ gives projective coordinates of a point of $\KK$, and the corresponding $\mathbf{w}$ gives projective coordinates of its image under $[\sqrt 5]$. 

For step 4 we use this set of pairs $T$ to construct a system of linear equations that the coefficients of the $\hat\varphi_i$ satisfy. For this, we must first know or guess the degree $N$ of the polynomials $\hat \varphi_i$. In our case, it seems that $N = 5$ always works. 
Consider the set $\mathfrak{m}$ of the monomials of degree $N$ in four variables $x_0,x_1,x_2,x_3$. The polynomials $\hat\varphi_i$ we want to find, take the form
\[ \hat\varphi_i = \sum_{\mu \in \mathfrak{m}} a_{i,\mu} \cdot \mu\]
for some unknown coefficients $a_{i,\mu}$. For each pair $(\mathbf{v}, \mathbf{w}) \in T$, writing $\mathbf{v} = (v_0, v_1,v_2,v_3)$ and $\mathbf{w} = (w_0,w_1,w_2,w_3)$, we have the relation
\[ [\widehat\varphi_0(\mathbf{v}):\widehat\varphi_1(\mathbf{v}):\widehat\varphi_2(\mathbf{v}):\widehat\varphi_3(\mathbf{v})] = [w_0: w_1: w_2: w_3]\]
as points of $\mathbb{P}^3$. Thus, for each pair $(\mathbf{v}, \mathbf{w})$ there is a non-zero constant $\lambda_{\mathbf{v}}$ such that 
\begin{equation}\label{abstract_equations}
    \sum_{\mu \in \mathfrak{m}} a_{i, \mu} \cdot \mu(\mathbf{v}) = \lambda_{\mathbf{v}} w_i 
\end{equation} 
holds for $i = 0,1,2,3$. This defines a linear system of equations in $4\cdot \#\mathfrak{m} + \#T$ unknowns (namely the $a_{i,\mu}$ and the $\lambda_{\mathbf{v}}$) with a total of $4 \cdot \#T$ linear relations between them. 

 To solve this system in MAGMA, we construct the matrix $M$ given by $M_{j,k}:={\mu}_k(\mathbf{v_j})$ where $\mathfrak{m} = \{\mu_k\}_{k = 1,2,\ldots, \#\mathfrak{m}}$ and $T = \{(\mathbf{v}_j, \mathbf{w}_j)\}_{j = 1,2, \ldots, \#T}$. Furthermore, build the diagonal matrices $\Delta^i$ using the $i$-th coordinate of each image point $\mathbf{w}_j = ((w_j)_0, (w_j)_1,(w_j)_2,(w_j)_3)$, that is, $\Delta^i_{j,j}=(w_j)_i$ and $\Delta^i_{j,j'}=0$ if $j\neq j'$.  With this we obtain the system:
 \[
 \mathcal{L}:=
   \begin{bmatrix}
M & \textbf{0} & \textbf{0} & \textbf{0} &-\Delta^0 \\
\textbf{0} & M & \textbf{0} & \textbf{0} &-\Delta^1 \\
\textbf{0} & \textbf{0} & M & \textbf{0} &-\Delta^2 \\
\textbf{0} & \textbf{0} & \textbf{0} & M &-\Delta^3 \\
  \end{bmatrix}.
\]
The kernel of $\mathcal{L}$ gives the coefficients $a_{i,\mu}$ and the constants $\lambda_{\mathbf{v}}$ satisfying Equation \ref{abstract_equations}. But note that not every solution is useful, as there will be solutions where one or more of the $\hat \varphi_i$ are identically zero on $\mathcal{K}$. For example, there is always the trivial solution where all variables are zero, but this solution is never useful. One should choose an element of the kernel for which each $\hat \varphi_i$ is non-zero on $\mathcal{K}$. This can be read off from the $\lambda_{\mathbf{v}}$: these must be non-zero. Note that $[\sqrt 5]$ and $[-\sqrt 5]$ induce the same map $\KK \to \KK$, so a Galois argument shows that $[\sqrt 5]\colon \KK \to \KK$ is defined over $\mathbb{Q}$. Hence, one should get a solution whose coefficients are in $\mathbb{Q}$. After scaling, these coefficients can be taken to be coprime integers. 

For step 5, one should check that the $\hat\varphi_i$ are coprime. If not, a common factor can be divided out. This happens only if $N$ is chosen too large.

Step 6 is a check to ensure that the polynomials $\hat\varphi_i$ indeed represent $[\sqrt 5]$. This check is necessary, because it is theoretically possible that the set of points $S$ used for the interpolation is not `generic' enough: if all points of $S$ happen to map into a curve on $\mathcal{K}$ of low degree, then the equations $\hat\varphi_i$ are only guaranteed to be correct on this curve and not on all of $\mathcal{K}$. Therefore, we check that the polynomials $\hat\varphi_i$ act correctly on a generic point on $\mathcal{K}$. To do this, we then consider the hyperelliptic curve $\mathcal{H}_F$ over $F:=\mathbb{Q}(\ja)$ given by the equation $y^2 = x^5 + h$, its Jacobian $\mathcal{J}_F$ and its associated Kummer surface $\KK_F$. By construction there is a generic point $P = [(x_1,y_1) + (x_2, y_2) - 2\infty] \in \mathcal{J}_F(F)$. Using Equation \ref{sq5}, we can then compute $Q = [\sqrt 5]P \in \mathcal{J}_F(F)$. With this we can check that
\[ [ \hat\varphi_0(\kappa(P)) :  \hat\varphi_1(\kappa(P)) :  \hat\varphi_2(\kappa(P)) :  \hat\varphi_3(\kappa(P))] = \kappa(Q) \]
as points in $\mathbb{P}^3(F)$. If this is the case, then the polynomials $\hat\varphi_i$ correctly represent the action of $[\sqrt 5]$. If not, one has to start with a larger set $S$ in step 1. \\
One can construct the function field of $F$ of $\ja$ in MAGMA via the description
\[F = \operatorname{Frac}\Big( \mathbb{Q}(\zeta)[A,B,C,D]/(\Psi_1,\Psi_2) \Big), \]
where $\Psi_1$ and $\Psi_2$ are the polynomials in $A,B,C,D$ satisfying the congruence
\begin{equation*}
X^5 +h  - (CX+D)^2\equiv \Psi_1 X + \Psi_2\bmod X^2+AX +B.
\end{equation*}
This congruence is used since it is easy to check that for all divisors of $\ja$ in Mumford coordinates $\langle X^2+AX+B,CX+D\rangle$ as in Equation \ref{MumfordGeneric} for any genus $2$ $\HC$ given by $Y^2=f(X)$, one has that $X^2+AX+B\mid f(X)-(CX+D)^2$.  \\ 
 
We implemented this procedure in MAGMA. Our implementation can be found on GitHub, see \cite{githubmagma}. As an example, we obtained the following polynomials representing the map $[\sqrt{5}]\colon \mathcal{K}\to\mathcal{K}$ for the Kummer surface $\mathcal{K}$ associated to the hyperelliptic curve $\HC$ given by $y^2=x^5 + 2$.

\begin{align}\label{equations_example_sqrt5}
\begin{split}
    \hat{\varphi}_0 :=\, & 320x_0^3x_1^2 + 80x_0^2x_1x_3^2 + 40x_0^2x_2^2x_3 + 80x_0x_1^2x_2x_3 -
        120x_0x_1x_2^3 \\
        & + 5x_0x_3^4 + 40x_1^3x_2^2 + 10x_1x_2x_3^3 +
        10x_2^3x_3^2\\
\hat{\varphi}_1:=\, & 640x_0^3x_1x_2 - 320x_0^2x_1^3 - 160x_0^2x_2x_3^2 + 120x_0x_1^2x_3^2 -
        40x_0x_1x_2^2x_3\\
        & + 120x_0x_2^4 - 80x_1^3x_2x_3 + 40x_1^2x_2^3 -
        5x_1x_3^4 + 10x_2^2x_3^3\\
\hat{\varphi}_2:=\, & 320x_0^3x_2^2 - 320x_0^2x_1^2x_2 - 40x_0^2x_3^3 - 320x_0x_1^4 -
        200x_0x_1x_2x_3^2 \\ 
        & +40x_0x_2^3x_3 + 80x_1^3x_3^2 - 120x_1^2x_2^2x_3 +
        5x_2x_3^4\\
\hat{\varphi}_3:=\, & 512x_0^5 + 320x_0^2x_1^2x_3 + 320x_0^2x_1x_2^2 - 40x_0x_1x_3^3 -
        360x_0x_2^2x_3^2\\
        &+ 64x_1^5 - 40x_1x_2^3x_3 + 24x_2^5 + x_3^5.
\end{split}
\end{align}
We have uploaded formulas for $[\sqrt 5]$ for various values of $h$ to GitHub, see \cite{githubkumendos}.

\subsection{Computation of \texorpdfstring{$P_0 = 4m^2Q_0$}{P0 = 4m2Q0}} \label{subsection_starting_vector}
The other ingredient that is needed to make \cref{main_theorem} into an algorithm is the point $P_0 = 4m^2 Q_0$, or rather, its image in $\KK$. This image will be the starting vector for the iterative application of $[\sqrt 5]$. Again, we use MAGMA to obtain this. Note that instead of computing $\kappa(4m^2 \cdot Q_0)$ directly, it is more efficient to compute $4m^2 \kappa(Q_0)$, i.e. first take the image of $Q_0$ in $\KK$ and then multiply by $4m^2$. The result is a vector of 4 projective coordinates with coprime coefficients in $\mathbb{Z}$. For example in MAGMA:
\begin{verbatim}
> alpha := -1; beta := 1; m := 1; h := beta^2 - alpha^5;
> J := Jacobian(HyperellipticCurve([1,0,0,0,0,h]));
> K := KummerSurface(J);
> K;
Kummer surface of Jacobian of Hyperelliptic Curve defined 
by y^2 = x^5 + 2 over Rational Field
> Q0 := elt<J| Polynomial([-alpha,1]),Polynomial([beta]),1>;
> Q0;
(x + 1, 1, 1)
> 4*m^2*K!Q0;                                       
(2624400 : -3559904 : 1744784 : 4190401)
\end{verbatim}

%\todo{Maybe we can include a sample MAGMA code? Its only a few lines. What do you think? -MP}\todoblue{what do you think of this?}\todo{Nice. I changed it a little bit to show more directly how the parameters in the text correspond to the code.}

We note that the size of this vector $\kappa(P_0)$ can be estimated using the theory of heights. Namely, consider the (logarithmic) canonical height $\hat h(Q_0)$ of $Q_0$. Then the canonical height of $P_0$ is $\hat h(P_0) = 16m^4 \cdot \hat h(P_0)$ because of the quadratic behaviour of the canonical height. Thus, one expects that also $\kappa(P_0)$ has logarithmic height approximately $16m^4 \cdot \hat h(Q_0)$, so the largest of the four coordinates of $\kappa(P_0)$ (after scaling so that the coordinates are coprime integers) should have absolute value around $\exp(16m^4 \cdot \hat h(Q_0))$, so just about $16m^4 \cdot \hat h(Q_0) / \log(10)$ decimal digits. 

For $Q_0 = (-1,1)$, so $h = 2$, we have $\hat h(Q_0) \approx 1.0279805$. For $m = 1$ we expect coordinates of size $10^{7.1}$, and we find
\[\kappa(4P) = [ 2624400 : -3559904 : 1744784: 4190401 ] \]
with indeed $8$ digits as expected. For $m = 2$, we expect coordinates of size $10^{114.3}$, and we find
\begin{align*}
\kappa(16P) = [ & 4046394669688530407248378946538416871445705653541548795862\\
                & 35105531112025243056923459621450802999130059192965945600  \quad :\\
                & 1517458072687990649583893248327329169920989863562377996111\\
                & 86211315039106660124123347467785740933508167817531798016 \quad : \\
                & -519775702244808047789789255873896726222838826011524190361\\
                & 702788673015740605567989171463669584295088827451720640256 \quad : \\
                & 7706059316740568145063375589890362492447388361047523236626\\
                & 20042074686653408771113007590496528284727186389858585601  ]
\end{align*}
which has coordinates with 114 decimal digits each. For $m = 3$, one expects 579 digits per coordinate, and indeed we get coordinates of this size. This illustrates how quickly the size of the starting vector grows with $m$: the number of digits grows with the fourth power of $m$. It also illustrates that one should choose $Q_0 = (\alpha, \beta)$ in such a way that its canonical height is small, in order to obtain a small starting vector $\kappa(P_0)$. 

We have computed $\kappa(P_0)$ for various choices of $Q_0$ and $m$. The results are available on GitHub, see \cite{githubstartvectors}.

\subsection{The iteration} \label{subsection_iteration}
Using the explicit representations of $[\sqrt 5]\colon \KK \to \KK$ and of $\kappa(P_0) = 4m^2\kappa(Q_0)$ obtained in the previous subsections, we can implement \cref{main_theorem} as an algorithm. This step does not need specialized mathematical software, and can be done in a general purpose programming language like C or Python. The procedure in pseudocode is shown in \cref{primality_pseudocode}.

We explain the steps. Line numbers 2 through 8 check that $h$ and $\lambda = \lambda_{m,n}$ are coprime, as this is used in the theory (if $h$ and $\lambda$ are not coprime, then the hyperelliptic curve $\mathcal{H}$ will fail to have good reduction at some primes dividing $\lambda$). Of course, if this finds a non-trivial factor of $\lambda$ then we are immediately done: $\lambda$ is composite. So the only inconclusive case here occurs if $h$ is a multiple of $\lambda$. Of course this should not happen in practice, as $h$ is usually taken small while $\lambda$ is big. 

From line $9$ on the core of the algorithm starts. We begin the iteration by taking the point $\kappa(P_0)$, considered as a vector consisting of four coprime integers, and take each component modulo $\lambda$. This is $\mathbf{v}_0$. After this, we recursively compute $\mathbf{v}_{r}$ by applying the polynomials $\hat\varphi_i$ to $\mathbf{v}_{r-1}$ and reducing the result modulo $\lambda$ again. This computes the image of $\kappa([\sqrt 5]^{r} P_0)$ modulo $\lambda$. At each step, we check whether $[\sqrt 5]^{r} P_0 = 0$ in $\mathcal{J}$ by checking if $\mathbf{v}_{r}$ is the point $(0 : 0 : 0 : 1)$ projectively. This is because the map $\kappa: \mathcal{J} \to \KK$ sends $0$ to $(0:0:0:1)$, and $0$ is the only point in $\mathcal{J}$ mapping to $(0:0:0:1)$. 

If after $2n$ iterations the point $(0:0:0:1)$ is not reached, then $\lambda$ is composite, per \cref{main_theorem}.1. We return this result in line 19. If we reach $(0:0:0:1)$ in at most $4 \log_5(\sqrt[4]{\lambda} + 1)$ steps, then the primality test is inconclusive: this is case 3 of \cref{main_theorem}. In all other cases, we are in the second case of \cref{main_theorem}, and we have to decide between cases 2a and 2b of that theorem. This means we have to check whether $\mathbf{v}_{r-1}$ is projectively equal to $(0:0:0:1)$ modulo some prime $p$ dividing $\lambda$. This is done by computing the gcd $d[i]$ of the first three components of $\mathbf{v}_{r-1}$ with $\lambda$. If $d[0]$, $d[1]$ or $d[2]$ is a non-trivial factor of $\lambda$, then $\lambda$ is not prime. Note that it is possible that one or two of the $d[i]$ are equal to $\lambda$, corresponding to $\mathbf{v}_{r-1}[i]$ being zero modulo $\lambda$, but it is not possible that all three are equal to $\lambda$. Therefore, if none of the $d[i]$ give a non-trivial factor of $\lambda$, then at least one of the $d[i]$ is 1, and so $\mathbf{v}_{r-1}$ is not $(0:0:0:1)$ modulo any primes $p$ dividing $\lambda$. By \cref{main_theorem}.2, we then conclude that $\lambda$ is prime. \\
 \\
\SetKwInput{KwData}{INPUT}
\SetKwInput{KwResult}{OUTPUT}
\scalebox{1.0}{\begin{minipage}{\textwidth}
\begin{algorithm}[H]\label{thealgorithm}
\SetAlgoLined
\LinesNumbered

 \KwData{$m, n \in \mathbb{N}$ with $m^2<\tfrac{(\sqrt{5}^n - 1)^4 + 1}{4\cdot 5^n}$, $h$, 
         polynomials $\hat \varphi_0,\ldots, \hat\varphi_3$ obtained from Subsection \ref{subsection_sqrt5}, 
         $\kappa(P_0)$ from Subsection \ref{subsection_starting_vector}}
 \KwResult{\textbf{prime} if $\lambda_{m,n}$ is prime, \textbf{composite} or \textbf{unknown}}
 $\lambda := 4m^2 5^n - 1$ \\
 $d := \gcd(h, \lambda)$ \\
 \If{$1 < d < \lambda$}{
    return \textbf{composite (factor $d$)}
 } 
 \If{$d = \lambda$}{
    return \textbf{unknown}  /* Choose a different $(\alpha,\beta)$ such that $\lambda\nmid h$ */
 }
 $\mathbf{v}_0:= \kappa(P_0) \mod \lambda$ \\
 reached\_identity := \textbf{false} \\
 \For{$r=1,\ldots ,2n$}{
  $\mathbf{v}_{r}$:= $(\hat{\varphi}_0(\mathbf{v}_{r-1}),\hat{\varphi}_1(\mathbf{v}_{r-1}),\hat{\varphi}_2(\mathbf{v}_{r-1}),\hat{\varphi}_3(\mathbf{v}_{r-1})) \bmod \lambda$
  
  \If{$\mathbf{v}_{r}[0] = \mathbf{v}_{r}[1] = \mathbf{v}_{r}[2]= 0$} {
    reached\_identity := \textbf{true} \\
    \textbf{break}
  }
 }
 \If{{\upshape reached\_identity = \textbf{false}}}{
   return \textbf{composite}
   }
 \If {$r>\tfrac{4\log(\sqrt[4]{\lambda_{m_0,n}} + 1)}{\log(5)}$} {
   \For{$i = 0,\ldots, 2$}{
     $d[i] := \gcd(\mathbf{v}_{r-1}[i],\lambda)$ \\
      \If {$1 < d[i] < \lambda$} {
        return \textbf{composite (factor $d[i]$)} 
      }
  }
  {
    return \textbf{prime}
  }
}
{return \textbf{unknown} /* retry with another $\alpha,\beta$ (rebuild $\widehat{\varphi}$ and $\kappa(P_0)$) */} 
 \caption{{\textbf{Primality test for $\lambda_{m,n}:=4m^2 5^n -1 $.}}} \label{primality_pseudocode}
\end{algorithm}
\end{minipage}}

\subsection{Implementation}
We implemented the pseudocode above in Python 3. The code can be found on GitHub \cite{githubpython}. We used this script, running on Python 3.6.4 in Darwin 18.7.0 x86\_64 (macOS Mojave 10.14.6) Intel Core M 1.2 GHZ, to compute for $m \in \{1,3,7,11\}$ the values $n < 500$ with $m^2<\tfrac{(\sqrt{5}^n - 1)^4 + 1}{4\cdot 5^n}$ for which $\lambda_{m,n}$ is prime. The results are in the following table. The column $n_0$ gives the smallest integer value for $n$ for which $m^2<\tfrac{(\sqrt{5}^n - 1)^4 + 1}{4\cdot 5^n}$ holds, i.e. the smallest $n$ for which the primality test applies. The columns $Q_0$ and $h$ indicate the curve and starting points used for the particular value of $m$. We used the starting points $Q_0$ from different hyperelliptic curves to demonstrate several $[\sqrt{5}]$ formulas, every row took 60 seconds in average to compute.  
%\todo[inline]{Do you think we should comment on why we have chosen these values for $\alpha$ and $\beta$? I think a reader might wonder why we have changed the value of $Q_0$ between different $m$. -MP}

\begin{table}[h]
    \centering
    \begin{tabular}{|l|l|l|l|l|}
        \hline
        $m$ & $n_0$ & $n \in \{n_0, \ldots, 500\}$  such that $\lambda_{m,n}$ is prime  & $Q_0 = (\alpha, \beta)$ & $h$ \\
        \hline
         $1$ & 2 & $3^*,9,13,15,25,39,69,165,171,209,339$ & (1,2), (-1,3)* & $3,10^*$\\
         \hline
        $3$ & 3 & $7,39$  & (-1,1) & $2$\\
        \hline
        $7$ & 4 & $39,53$  & (2,1) & $-31$\\
        \hline 
         $11$ & 4 & $19,55,89,91,119,123,177,225,295$ & (-1,3) & $10$\\
         \hline       
    \end{tabular}
    \caption{Implementation example}
    \label{tab:my_label}
\end{table} 
%\todo{I was thinking about removing the 'Time' column, because it doesn't add so much and it makes our table to wide for the margins. What do you think? -MP}

The one entry marked with $^*$ means that for that pair $(m,n)$ the test had to choose another pair $(\alpha, \beta)$ to determine primality successfully, because the initial choice of $Q_0$ lead to an indeterminate outcome. It is seen in the table that this only occurred once, for a very small value $n$. In \cref{probability-section} we will show that this is the expected behavior: the chance of reaching the indeterminate outcome decreases exponentially with $n$. 
 
All the $\kappa(P_0)$ points for each $m\in\{1,3,7,11\}$ and each curve $y^2=x^5+h$ where $h\in \{2,3,10,-31\}$ can be found in \cite{githubstartvectors} (Python). Explicit formulas for the $[\sqrt{5}]$ endomorphisms for each $h$ can be found in \cite{githubkumendos}. Furthermore, a MAGMA script to generate other choices of $\kappa(P_0)$ for different $m$ using other curves is in \cite{githubstartvectorsgenerator}. Their respective explicit $[\sqrt{5}]$ endomorphisms (or other endomorphisms) are calculated using \cite{githubmagma}. CSV files with other low-height vectors for several $m$ can be found on Github \cite{githubcsvvectors}.

\section{The probability of getting an indeterminate case} \label{probability-section}

\cref{primality_pseudocode} has an indeterminate case, corresponding to case 3 of \cref{main_theorem}. The goal of this section is to make precise and prove the statement that this indeterminate case is very rare, at least for the case that $\lambda_{m,n}$ is prime. Essentially, the idea is that only a very small fraction of pairs $(\alpha, \beta)$ lead to the indeterminate case, while all other pairs prove primality of $\lambda$. 

In this section, we fix $m$ and $n$, and we assume that $\lambda = \lambda_{m,n}$ is prime. We would like to know the probability that, starting from random $(\alpha, \beta)$ (random in the sense made precise below), the primality test fails to prove that $\lambda$ is prime. According to \cref{main_theorem}, this is equivalent to asking for the probability that the image of the point $(\alpha, \beta)$ in $\mathcal{H}(\FF_{\lambda})$ lies in $[\sqrt 5]^k(\mathcal{J}(\F_\lambda))$ for some $k \geq 2n - 4\log_5(\sqrt[4]{\lambda} + 1)$. Note that this only depends on $\alpha$ and $\beta$ modulo $\lambda$, so a sensible interpretation of ``random $(\alpha, \beta)$" is that the pair $(\alpha, \beta)$ is to be regarded as uniformly distributed along pairs of integers in $\{0,1,\ldots, \lambda-1\}$ such that $h = \beta^2 - \alpha^5$ is not $0$ modulo $\lambda$. Note that $h$, $\mathcal{H}$ and $\mathcal{J}$ all depend on $(\alpha, \beta)$, so they are also random variables.  

We prove the following.
\begin{theorem}\label{probabilitytheorem}
Assume that $\lambda > 100$ and that $\lambda$ is prime. Let $k \geq 2n - 4\log_5(\sqrt[4]{\lambda} + 1)$. Then 
\[ \mathbb{P}\big(\, (\alpha, \beta)  \in [\sqrt 5]^k\mathcal{J}(\mathbb{F}_\lambda) \,\big) \leq \frac{2m}{5^{n/2}}. \]
\end{theorem}
%\todoblue{Just to be sure, I think you made some kind of notational  comment on $(\alpha,\beta)\in\HC$ being also treated as $[(\alpha,\beta)-\infty]\in\ja$, right ? or do we need to recall it?}
Here, as in the rest of this paper, we regard $\mathcal{H}(\mathbb{F}_\lambda)$ as a subset of $\mathcal{J}(\mathbb{F}_\lambda)$ via the Abel-Jacobi map, i.e. the point $(\alpha, \beta)$ of $\mathcal{H}$ corresponds to the divisor class of $[(\alpha, \beta)] - [\infty]$ as a point on $\mathcal{J}$. %\todo{I added a clarifying line, what do you think of it? -MP}
\begin{proof}
The proof of this theorem will take up the rest of the section. 

We introduce the following sets:
\[ X = \{ P \in \mathcal{H}(\F_\lambda) : P \in [\sqrt 5]^k(\mathcal{J}(\F_\lambda)) \} \]
and
\[ Y = \{D \in \mathcal{J}(\F_\lambda) : \exists P, Q \in X: D = P+Q\}. \]
In words, $X$ is the set of $\F_\lambda$-points of $\mathcal H$ which are in the image of $[\sqrt 5]^k$ when regarded as an $\F_\lambda$-point of $\mathcal{J}$, and $Y$ is the set of elements of $\mathcal{J}(\F_\lambda)$ that can be written as the sum of two points in $X$. Since $X$ and $Y$ depend on $\mathcal{J}$ and $\mathcal{H}$, they should be regarded as random sets. Notice that the probability in the theorem is $\mathbb{P}(\, (\alpha, \beta)  \in X \,)$

Notice that $(\alpha, \beta)$ can be any element of $\mathcal{H}$ except the point at infinity, and all other points are equally likely. Therefore, if we have a bound $\#X \leq B$ for some constant $B$, then we get an upper bound for the probability that we are looking for, namely 
\begin{equation}\label{upperboundprobability}
    \mathbb{P}\big(\, (\alpha, \beta)  \in X \,\big) \leq \frac{B - 1}{\#\mathcal{H}(\F_\lambda) - 1}. 
\end{equation} 
The $-1$ in the numerator and denominator accounts for the fact that $(\alpha,\beta)$ cannot be the point at infinity (which is always in $X$). 

Therefore, we want to bound $\#X$. We first relate it to $\#Y$.
\begin{lemma}
We have 
\[ \#Y = \frac{1}{2}(\#X)^2.\]
\end{lemma}
\begin{proof}
We know the elements of $Y$ in terms of elements of $X$: there is the identity element $0$, there are $\#X-1$ elements of $Y$ of the form $P + 0$ with $P \in X \setminus \{0\}$, there are $\#X-2$ elements of $Y$ of the form $P + P$ with $P \in X \setminus \mathcal{J}[2](\F_\lambda)$ (by \cref{2torsionforprimality} there is exactly one non-trivial 2-torsion point in $\mathcal{J}(\F_\lambda)$, and since it is fixed by $[\sqrt 5]$ it is an element of $X$), $\#X - 2$ elements of the form $P_0 + P$ where $P_0$ is the non-trivial 2-torsion point and $P \in X$ is not $2$-torsion, and finally there are $(\#X-2)(\#X-4)/2$ elements of the form $P + Q$, where $P$ and $Q$ are in $X$, $P$ and $Q$ are not 2-torsion, and $P \neq \pm Q$. By the uniqueness of the decomposition $D = P+Q$ for $D \neq 0$, we find that 
\[ \#Y = 1 + (\#X-1) + (\#X-2) + (\#X-2) + \frac{(\#X-2)(\#X-4)}{2} = \frac{1}{2}(\#X)^2.\qedhere\]
\end{proof}

\begin{lemma}\label{upperboundX}
We have 
\[ \# X \leq \sqrt{\frac{2\# \mathcal{J}(\F_\lambda)}{5^k}} = \frac{\sqrt{2}(\lambda + 1)}{5^{k/2}} \leq \frac{4\sqrt{2}m(\lambda + 1)}{5^{n/2}} = 4\sqrt{2}\cdot m^35^{n/2} \]
\end{lemma}
\begin{proof}
By construction, we have $Y \subseteq [\sqrt 5]^k(\mathcal{J}(\F_\lambda))$. From \cref{sqrt5-module} the kernel of $[\sqrt 5]^k$ has size $5^k$, so the image of $[\sqrt 5]^k$ has index $5^k$. This gives $\#Y \leq \#\mathcal{J}(\F_\lambda)/5^k$. Since $\#X = \sqrt{2\#Y}$ and $\#\mathcal{J}(\F_\lambda) = (\lambda+1)^2$, the first inequalities follow. 

We now estimate $5^{k/2}$. Since $k \geq 2n - 4\log_5(\sqrt[4]{\lambda} + 1)$, we get that 
\[5^{k/2} \geq 5^n \cdot (\lambda^{1/4} + 1)^{-2} 
= 5^n (\lambda+1)^{-1/2}\cdot \frac{(1 + \lambda^{-1})^{1/2}}{(1 + \lambda^{-1/4})^2}
\geq \frac{1}{2}\cdot 5^n\cdot (\lambda+1)^{-1/2}, \]
where we use that for $\lambda > 100$ we have $\tfrac{(1 + \lambda^{-1})^{1/2}}{(1 + \lambda^{-1/4})^2} > 1/2$. Filling in $\lambda = 4m^2 5^n - 1$ we get
\[ 5^{k/2} \geq \frac{1}{2}\cdot 5^n \cdot (4m^25^n)^{-1/2} 
 = \frac{5^{n/2}}{4m}. \]
Filling this in gives the remaining estimate. 
\end{proof}
We want to combine the upper bound \cref{upperboundX} with equation \ref{upperboundprobability}. All we need is to know $\#\mathcal{H}(\F_\lambda)$. This is 1 + the number of solutions to the equation $y^2 = x^5 + h$ in $\F_\lambda$. But since $\lambda$ is $4$ mod 5, the fifth power map is a bijection $\F_\lambda \to \F_\lambda$, and so for every $y \in \F_\lambda$ there is a unique $x \in \F_\lambda$ such that $y^2 = x^5 + h$. Therefore $\#\mathcal{H}(\F_\lambda) = \lambda + 1 = 4m^25^n$. Filling all this into Equation \ref{upperboundprobability}, we get
\begin{align*}
     \mathbb{P}\big(\, (\alpha, \beta)  \in X \,\big) &\leq \frac{4\sqrt{2}\cdot m^35^{n/2} - 1}{4m^25^n-1} 
     \\&= \frac{m}{5^{n/2}} \cdot \sqrt{2}\cdot \frac{1 - (4\sqrt{2} m^3 5^{n/2})^{-1}}{1 - (4m^25^n)^{-1}} \\
     &\leq \frac{m}{5^{n/2}} \cdot \frac{4}{3}\sqrt{2}.
\end{align*} 
\cref{probabilitytheorem} now follows since $\frac{4}{3}\sqrt{2} < 2$. 
\end{proof}

A short computation shows that $m^2 < \tfrac{(\sqrt 5^n - 1)^4 + 1}{4\cdot 5^n}$ implies $4m^2 < 5^n$. Therefore, for the pairs $m,n$ that we look at, the upper bound in \cref{probabilitytheorem} is non-trivial. In particular, for such $m,n$ with $\lambda_{m,n}$ prime there always exist $(\alpha, \beta)$ which proves primality. Thus, in this case the algorithm finishes in finite time. Moreover, \cref{probabilitytheorem} shows that the probability of failure of a given starting point decreases exponentially with $n$. Already for $n \sim 100$, the probability of failure is so small that it seems unlikely to ever occur in practice.

\bibliographystyle{elsarticle-harv}
\bibliography{sample.bib}

\end{document}